\documentclass{amsart}
\usepackage{amsmath}
\usepackage{amsxtra}
\usepackage{amscd}
\usepackage{amsthm}
\usepackage{amsfonts}
\usepackage{amssymb}
\usepackage{mathdots}
\usepackage{bbm}
\usepackage{graphicx}
\usepackage{cite}
\usepackage[bookmarksnumbered, colorlinks, plainpages]{hyperref}
\usepackage[utf8]{inputenc}
\usepackage{color}
\usepackage{enumerate}
\newtheorem{theorem}{Theorem}[section]

\newtheorem{open}{Problem}[section]
\newtheorem{lemma}[theorem]{Lemma}

\theoremstyle{definition}

\newtheorem{remark}[theorem]{Remark}

\usepackage[paperwidth=17.10cm, paperheight=24.60cm, left=1.5cm, right=1.5cm, top=2cm, bottom=2cm]{geometry}

\usepackage{etoolbox}
\patchcmd{\section}{\scshape}{\large}{}{}
\patchcmd{\subsection}{\bfseries}{\normalfont}{}{}
\patchcmd{\subsubsection}{\itshape}{\normalfont}{}{}

\makeatletter
\def\@setauthors{%
	\begingroup
	\def\thanks{\protect\thanks@warning}%
	\trivlist
	\centering\footnotesize \@topsep30\p@\relax
	\advance\@topsep by -\baselineskip
	\item\relax\normalsize 
	\author@andify\authors
	\def\\{\protect\linebreak}%
	\authors%
	\ifx\@empty\contribs
	\else
	,\penalty-3 \space \@setcontribs
	\@closetoccontribs
	\fi
	\endtrivlist
	\endgroup
}
\def\@settitle{\begin{center}%
		\baselineskip14\p@\relax
		\bfseries\large
		\@title
	\end{center}%
}
\makeatother

\usepackage{tikz,xcolor,hyperref}

\definecolor{lime}{HTML}{A6CE39}
\DeclareRobustCommand{\orcidicon}{%
	\begin{tikzpicture}
		\draw[lime, fill=lime] (0,0) 
		circle [radius=0.16] 
		node[white] {{\fontfamily{qag}\selectfont \tiny ID}};
		\draw[white, fill=white] (-0.0625,0.095) 
		circle [radius=0.007];
	\end{tikzpicture}
	\hspace{-2mm}
}

\usepackage{titlesec}
\titleformat{\section}
{\normalfont\large\bfseries}
{\thesection}{1em}{}
\titleformat{\subsection}
{\normalfont\large\bfseries}
{\thesubsection}{1em}{}

\newcommand{\orcidSon}{\href{https://orcid.org/0000-0002-9560-6392}{\orcidicon}}

\begin{document}
	\title[On sums and products of diagonalizable matrices]{On sums and products of diagonalizable matrices over division rings}
	\author[Tran Nam Son]{Tran Nam Son\textsuperscript{1,2,\ddag}}
	
	\maketitle
	
	\begin{center}
		{\small 
			\textsuperscript{1}Faculty of Mathematics and Computer Science,\\ University of Science,\\ Ho Chi Minh City, Vietnam \\
			\textsuperscript{2}Vietnam National University,\\ Ho Chi Minh City, Vietnam\\
			\textsuperscript{\ddag}trannamson1999@gmail.com,\\\orcidSon \url{https://orcid.org/0000-0002-9560-6392}
		}
	\end{center}
	
	\begin{abstract}
This paper aims to continue the studies initiated by Botha in [\textit{Linear Algebra Appl.} \textbf{273} (1998), 65–82; \textit{Linear Algebra Appl.} \textbf{286} (1999), 37–44; \textit{Linear Algebra Appl.} \textbf{315} (2000), 1–23] by extending them to matrices over noncommutative division rings. In particular, we show that every such matrix can be written as either a sum or a product of two diagonalizable matrices. The number $2$ is not valid under mild conditions on the center, similar to those in Botha’s work on fields. By applying this result and other results obtained so far, we latter establish some Waring-type results for matrices.
	\end{abstract}
	

	\section{Introduction}\label{intro}
	
	Matrix diagonalization is a foundational topic in linear algebra, with numerous applications. It transforms a square matrix into a diagonal form through similarity transformations, where the non-zero entries are restricted to the main diagonal. This simplification makes operations like multiplication and exponentiation much easier. Diagonalization is particularly helpful in solving linear differential equations, finding energy eigenstates in quantum mechanics, and aiding transformations in computer graphics. Its ability to simplify complex problems makes it an invaluable tool across many fields. Even if a matrix is not diagonalizable, it is still possible to ``do the best one can" by finding a matrix with similar properties, where eigenvalues appear on the leading diagonal, and either ones or zeroes occupy the superdiagonal - a form known as the Jordan normal form. 
	
	On a different note, although some matrices may not be diagonalizable, it is interesting, as pointed out by Botha in \cite{Pa_Bo_98, Pa_Bo_99, Pa_Bo_20}, that every matrix over a field containing at least four elements can be expressed as a sum or product of two diagonalizable matrices. This result is particularly striking for two reasons: first, it shows that sums or products of diagonalizable matrices are sufficient to generate the entire matrix algebra. Second, it reveals that there exists a uniform bound on the number of summands or factors required. Moreover, this finding provides an avenue for addressing the modern Waring problem in various ways.
	
	The classical Waring problem concerns the representation of natural numbers as sums of $k$-powers, where $k$ is a positive integer. Modern versions explore similar questions in settings with noncommutative structures, meaning that the Waring problem, classically posed for integers, naturally considers to an algebra \( A \) over a field \( K \), where \( A \) need not be commutative. In the context of matrix algebras, the Matrix Waring Problem asks for the smallest integer \( m \) such that every element of the algebra \( A = \mathrm{M}_n(K) \) can be written as a sum of \( m \) \( k \)-th powers. This problem has led to numerous interesting results over time \cite{Pa_Ki_22,Pa_Ki_24,Pa_Ki_24_1,Pa_Ka_13 ,Pa_Ka_24}.

	The contribution of this paper is to further investigate Botha's results in the context of matrices over a division ring. In particular, we prove that every matrix over a noncommutative division ring can also be represented as either a sum or a product of two diagonalizable matrices. In this paper, we show that under certain conditions, specifically, for instance, considering radicable division rings, every matrix over the division ring can be expressed as a sum, product, or linear combination of two $k$-th powers. Moreover, this result provides insight into the study of images of diagonal polynomials, an area that has recently attracted considerable attention.
	
	For convenience, we briefly review some of the notations that will be used throughout the paper. Recall that a square matrix over a ring \( D \) is \textit{diagonalizable} if it can be expressed as similar to a diagonal matrix over \( D \).  In this paper, we will focus on division rings, and unless stated otherwise, \( D \) will refer to an arbitrary noncommutative division ring.    We use the standard notations \(\mathrm{M}_n(D)\) for the ring of \(n \times n\) matrices over a division ring \(D\) where $n$ is a positive integer, and \(\mathrm{GL}_n(D)\) for the general linear group of \(n \times n\) matrices over \(D\), respectively. Moreover, the symbol $\mathrm{M}_{t\times m}(D)$ denotes the set of all $t\times m$ matrices over $D$ where $t$ and $m$ are positive integers. Additionally, we write $\mathrm{LT}_n(D)$ and $\mathrm{UT}_n(D)$ as sets of lower and upper unitriangular matrices respectively, that is, $\mathrm{LT}_n(D)$ (resp., $\mathrm{UT}_n(D)$) consists of lower triangular matrices (resp., upper triangular matrices) whose diagonal entries are~$1$. Furthermore, $\mathrm{SL}_n(D)$ stands for the commutator subgroup of $\mathrm{GL}_n(D)$. For $a_1, \ldots, a_n \in D$, the notation $\mathrm{diag}(a_1, \ldots, a_n)$ refers to the diagonal matrix with entries $a_1, \ldots, a_n$ on its diagonal. 	As usual, $\mathrm{I}_n$ is the identity matrix in $\mathrm{M}_n(D)$. For $\lambda\in D$, the matrix $\lambda\mathrm{I}_n$ is the matrix $\mathrm{diag}(\lambda,\ldots,\lambda)\in\mathrm{M}_n(D)$.  A matrix in $\mathrm{M}_n(D)$ is a \textit{companion} matrix if it has the following form: $\begin{pmatrix}
		0& & &a_0 \\
		1& & &a_1\\
		&\ddots& &\vdots\\
		0& &1&a_{n-1}
	\end{pmatrix},$ where $a_0,a_1,\ldots,a_{n-1}\in D$. If $n=n_1+n_2+\ldots+n_t$ where $n_1,n_2,\ldots,n_t$ are positive integers, then a \textit{block diagonal matrix} \( A\in\mathrm{M}_n(D) \) has the form	$A =
	\begin{pmatrix}
		A_1 & 0 & \cdots & 0 \\
		0 & A_2 & \cdots & 0 \\
		\vdots & \vdots & \ddots & \vdots \\
		0 & 0 & \cdots & A_t
	\end{pmatrix},$ where \( A_i\in\mathrm{M}_{n_i}(D) \) for all \( i = 1, \ldots, t \) and \( 0 \) represents zero matrices of appropriate dimensions to match the block structure. This matrix is denoted as: $A = A_1 \oplus A_2 \oplus \cdots \oplus A_n \text{ or }A=\bigoplus_{i=1}^nA_i.$ If $A_i=A'$ for all $i=1,\ldots,t$, then we write $A$ as $\displaystyle A=\bigoplus_{n}A'$ for short. A matrix in $\mathrm{M}_n(D)$ is said to be \textit{traceless} if it has trace zero.

With what has been introduced so far, we proceed to discuss the main tools and outline the structure of the paper. The first topic, expressing matrices as sums and products of diagonalizable matrices, was previously studied by Botha in \cite{Pa_Bo_98,Pa_Bo_99,Pa_Bo_20}, and our work benefits from his contributions. We will follow Botha's approach by examining companion matrices. While it is well known that not all square matrices over a field are similar to a companion matrix, every square matrix over a field is similar to a block diagonal matrix  of companion matrices (see \cite[Theorem B-3.47]{Bo_Ro_15}). This result also holds for matrices over a division ring, though it may be less widely known. It is discussed in \cite[Chapter~3, Section~12, Page~50]{Bo_Ja_43} in the context of a single semi-linear transformation and later in \cite[Section 8.4, Page 505]{Bo_Co_85} by Cohn for pseudo-linear transformations. Additionally, it has been utilized by D. Ž. Djoković and J. G. Malzan in \cite{Pa_Djo_79}. Given this background, we aim to focus on decomposing companion matrices into sums or products of diagonalizable matrices.
	
	In the context of the Waring problem, understanding the base division ring is essential. For instance, in division rings whose multiplicative groups are radicable, every element can be expressed as a single \( k \)-th power. Hence, in this case, we investigate the Matrix Waring problem in the noncommutative setting, focusing on expressing matrices as sums of \( k \)-th powers. A natural question that arises in this study is whether the center of a radicable division ring must be infinite. To address this, we construct a counterexample. This construction benefits from the work of P. M. Cohn concerning EC-division rings. In particular, let \( K \) be a division ring of positive characteristic, and choose \( k \) to be a finite field contained in the center of \( K \). According to Cohn’s book, there exists an EC-division ring \( D \) that contains \( K \). Furthermore, for any positive integer \( m \) and any \( a \in D \), the equation $x^m - a = 0$ always has a solution in some extension of \( D \). Since \( D \) is an EC-division ring, such a solution must also exist within \( D \) itself. Notably, the center of \( D \) is precisely \( k \), providing the desired example. We sincerely thank a Mathematics Stack Exchange user in Question 5030291 for drawing our attention to this result.

In summary, the key components of the proofs in this paper include the theory of expressing matrices via diagonalizable matrices, the theory of polynomials over division rings, several commutativity theorems in division rings, the classification theory of division rings, and the theory of rational canonical forms of matrices over division rings, among others. While it is unsurprising that these theories play a valuable role in our study, the approach of integrating them within the proofs seems to be new.

	The organization of this paper, along with a selection of its results, is presented below:  Although Botha's work initially emphasizes products before sums, we, by our preference, choose to focus on sums before products. Accordingly, Section~\ref{sum} is dedicated to the decomposition of matrices over $D$ into sums of diagonalizable matrices over $D$. We show in Theorem~\ref{sum1} that every matrix over \( D \) can be written as a sum of two diagonalizable matrices over \( D \), provided the center of \( D \) contains at least three elements. Clearly, the number $2$ is the smallest possible value. When the center of \( D \) is the field of two elements, we prove in Theorem~\ref{sum2} that every matrix over \( D \) can be expressed as a sum of three diagonalizable matrices over \( D \), with the number $3$ being the minimal value. Additionally, we aim to find a necessary and sufficient condition for a matrix over \( D \) to be written as a sum of two diagonalizable matrices over \( D \). With this goal in mind, although we have not yet found the desired condition, we provide some insight into the problem by showing in Theorem~\ref{sum3} that every nilpotent matrix over \( D \) can indeed be expressed as a sum of two diagonalizable matrices over \( D \), regardless of the condition of the center of $D$.
	
	Following this, Section~\ref{product} shifts the focus to matrix decomposition as  products of diagonalizable matrices. This scenario introduces more complications, leading to a division into three subsections, depending on the characteristic of $D$ and the number of elements in its center (see Theorems~\ref{cha khac 2}, \ref{cha 2}, \ref{nil2} and \ref{enough}). To assist the readers and make referencing the next section easier, we finish this section with a concise summary in Theorem~\ref{section diagonalizable} of the results presented so far. 
	
	The final section, Section~\ref{Waring section}, focuses on the Waring problem. We begin our discussion with radicable division rings in Subsection~\ref{radicable}. By definition, every element of such a division ring can be expressed as a $k$-power for any positive integer $k$. Building on the results established in previous sections, we prove in Theorem~\ref{diagonal map} that every matrix over a radicable division ring with infinite center can be written as either a sum or a product of two $k$-powers. Furthermore, we highlight that the infiniteness of the center in this setting is a nontrivial property. While any radicable division ring that is finite-dimensional over its center must have the infinite center, we provide an illustrative example, as mentioned above in the first part of this subsection, to demonstrate this point. Leveraging this result, we further establish in Theorem~\ref{dia1} that every matrix can also be expressed as a linear combination of two $k$-powers, where the two scalar coefficients are arbitrary but nonzero. As a consequence, we deduce that the map induced by a diagonal polynomial in $m$ noncommuting variables is surjective for all $m \geq 2$. We have also refined certain results related to the Waring problem for groups in the setting of twisted group algebras. Specifically, in Theorem~\ref{twisted}, we show that every element in $\mathcal{U}(F^\tau G)'$, the derived subgroup of the unit group of the twisted group algebra $\mathcal{U}(F^\tau G)$, can be expressed as a product of two square elements whenever $G$ is a locally finite group and the base field $F$ has characteristic zero, provided that $F$ is either algebraically closed or real-closed.

In Subsection~\ref{algebraically}, we turn our attention to a specific class of radicable division rings, namely algebraically closed division rings. Here, we adopt the notion in the sense of Niven, Baer, and Jacobson, meaning that every one-sided polynomial equation in a central variable has a root. Our first main result in this direction, Theorem~\ref{central x}, establishes that if \( p \) is a nonconstant polynomial in the central variable \( x \) with coefficients in \( F \), the center of an algebraically closed division ring \( D \), then the image of \( p \) on \( D \) coincides with \( D \), i.e., \( p(D) = D \).  Building on this, we further show in Theorem~\ref{al} that if \( p \) is a polynomial in noncommuting variables \( x_1, \dots, x_m \) with coefficients in \( F \), where \( m \) is a positive integer, such that \( p \) has zero constant term and \( p(F) \neq \{0\} \), then \( p(D) = D \). Leveraging the results obtained thus far, we investigate the Waring problem in this setting. Specifically, in Theorems~\ref{al1} and~\ref{allll}, we establish that every matrix over an algebraically closed division ring \( D \) with infinite center \( F \) can be written as either a sum, a product, or a linear combination of two elements from \( p(\mathrm{M}_n(D)) \), where the two scalar coefficients are chosen arbitrarily in \( F \setminus \{0\} \). Furthermore, in the absence of the assumption that \( p(F) \neq \{0\} \), we show that if \( F \) has characteristic zero and \( \dim_F D < \infty \), and if \( p \) is not cyclically equivalent to an identity of \( \mathrm{M}_n(F) \), then every matrix in \( \mathrm{M}_n(D) \) can be expressed as a linear combination of five elements from \( p(\mathrm{M}_n(D)) \).  Additionally, Theorem~\ref{al1} extends certain results of M. Brešar and P. Šemrl \cite{Pa_Bre_23, Pa_Bre_23_1} to the noncommutative setting, incorporating modified assumptions as outlined in Theorem~\ref{al1}, while Theorem~\ref{allll} follows from \cite[Corollary~4.5]{Pa_Bre_23} and a description of division rings. Here, we say that two polynomials \( f \) and \( g \) in noncommuting variables are \textit{cyclically equivalent} if the difference \( f - g \) can be expressed as a sum of commutators.

	\section{Sums of diagonalizable matrices}\label{sum}
	
	In 2000, Botha investigated in \cite{Pa_Bo_20} the decomposition of matrices into sums of diagonalizable matrices and established an intriguing result: any matrix over a field with at least three elements can be written as a sum of two diagonalizable matrices. In the remaining case with only two elements, every matrix over such a field can be expressed as a sum of three diagonalizable matrices. Moreover, a matrix  can be expressed as a sum of two diagonalizable matrices if and only if it is similar to a matrix of the form  
	\[ 
	N \oplus \begin{pmatrix} 
		0 & X \\ 
		\mathrm{I} & 0 
	\end{pmatrix} \oplus (\mathrm{I} + M), 
	\]  
	where \( N \) and \( M \) are nilpotent, \( X \) and \( X + \mathrm{I} \) are nonsingular, and \( \mathrm{I} \) represents the identity matrix of suitable size.
	
	 In this section, we extend this result by addressing noncommutative division rings instead of fields. 	
	 To begin, we present the following remark, inspired by \cite[Lemma 1.1]{Pa_Bo_98}, which can be easily verified.
	 
	 \begin{remark} \label{remark di}
	 	\begin{enumerate}[\rm (i)]
	 		\item If \(A \in \mathrm{M}_n(D)\) is diagonalizable and similar to \(B \in \mathrm{M}_n(D)\), then \(B\) must also be diagonalizable.
	 		\item If $n=n_1+n_2+\ldots+n_k$ where $k\geq1$ is an integer and  $$A_1\in\mathrm{M}_{n_1}(D), A_2\in\mathrm{M}_{n_2}(D), \ldots,A_k\in\mathrm{M}_{n_k}(D)$$ are diagonalizable, then the diagonal block matrix $$A_1\oplus A_2\oplus\cdots\oplus A_k=\begin{pmatrix}A_1&0&\cdots&0\\0&A_2&\cdots&0\\\vdots &\vdots&\ddots&\vdots\\ 0&0&\cdots &A_k\end{pmatrix}\in \mathrm{M}_n(D)$$ is also diagonalizable.
	 	\end{enumerate}
	 \end{remark}

Without adopting Botha’s approach in \cite{Pa_Bo_20}, we can establish the expected result under the assumption that the center of the division ring contains sufficiently many elements. While we will later prove this result without imposing this condition, we present it here to introduce additional ideas that may be of interest to the reader.

Let \( D \) be a division ring with center \( F \), where \( F \) contains sufficiently many elements, and let \( n > 1 \) be an integer. It is straightforward to verify that any central matrix in \( \mathrm{M}_n(D) \) takes the form \( \lambda \mathrm{I}_n \) for some \( \lambda \in F \), making it trivially expressible as a sum of two diagonalizable matrices. Moreover, \cite[Proposition~1.8]{Pa_AmRo_94} establishes that any noncentral matrix in \( \mathrm{M}_n(D) \) is similar to a matrix whose main diagonal consists of \( 0, \dots, 0, d \) for some \( d \in D \). Notably, while \cite{Pa_AmRo_94} assumes that \( D \) is finite-dimensional over its center, the proof of \cite[Proposition~1.8]{Pa_AmRo_94} itself does not rely on this assumption.  

Now, let \( A \in \mathrm{M}_n(D) \) be a noncentral matrix. By \cite[Proposition~1.8]{Pa_AmRo_94}, there exists an invertible matrix \( P \in \mathrm{M}_n(D) \) such that \( P^{-1} A P \) has diagonal entries  
\[
\underbrace{0,0,\dots,0,}_{(n-1) \text{ times}} d
\]  
for some \( d \in D \). We then decompose \( P^{-1} A P \) as  
\[
P^{-1} A P = B + C,
\]  
where \( B \) is a lower triangular matrix with zero diagonal entries, and \( C \) is an upper triangular matrix whose diagonal consists of  
\[
\underbrace{0,0,\dots,0,}_{(n-1) \text{ times}} d.
\]  

Since the center \( F \) of \( D \) contains sufficiently many elements, we can choose distinct elements \( \alpha_1, \alpha_2, \dots, \alpha_{n-1} \in F \) and define  
\[
K = \mathrm{diag}(\alpha_1, \alpha_2, \dots, \alpha_{n-1},0).
\]  
This allows us to rewrite the decomposition as  
\[
P^{-1} A P = (B+K) + (-K+C).
\]  
By \cite[Lemma~2.1]{Pa_DuSo_25}, the matrix \( B + K \) is similar to \( K \). However, the presence of \( d \) in \( -K+C \) prevents us from applying \cite[Lemma~2.1]{Pa_DuSo_25} directly. Following a similar argument in the proof of \cite[Lemma~2.1]{Pa_DuSo_25}, we can show that \( -K+C \) is similar to  
\[
\mathrm{diag}(-\alpha_1, -\alpha_2, \dots, -\alpha_{n-1}, d),
\]  
provided that \( \alpha_1, \alpha_2, \dots, \alpha_{n-1}, d \) are distinct. This can indeed be ensured due to the abundance of elements in the center of \( D \). Thus, the proof is complete.  

In the following subsections, we proceed without assuming that the center of \( D \) contains sufficiently many elements.

	\subsection{Division rings with center containing at least three elements}
	
	The subsequent key step of Botha is to determine whether the matrix $\begin{pmatrix}
		B & \alpha \\
		0 & a
	\end{pmatrix}$ is similar to the matrix $\begin{pmatrix}
		B & 0 \\
		0 & a
	\end{pmatrix},$ where \( B \in \mathrm{M}_n(D), \alpha\in\mathrm{M}_{n \times 1}(D) \),  and \( a \in D \). To approach this question, we start as follows. If there exists $y\in\mathrm{M}_{n \times 1}(D)$ such that $$\begin{pmatrix}
		\mathrm{I}_n&y\\
		0&1
	\end{pmatrix}\begin{pmatrix}
		B&\alpha\\
		0&a
	\end{pmatrix} \begin{pmatrix}\mathrm{I}_n&y\\
		0&1
	\end{pmatrix}^{-1}=\begin{pmatrix}
		B&0\\
		0&a
	\end{pmatrix},$$ then $\alpha=By-ya$. Therefore, the question reduces to whether there exists  $y\in\mathrm{M}_{n \times 1}(D)$   such that  $\alpha=By-ya$. This equation can be viewed as  the Sylvester equation in the context of matrices over a division ring.  A result of Bolotnikov in \cite[Lemma 3.1]{Bo_Bo_22} provides a sufficient condition for solving this equation, stated as follows. Although the statement originally provides the solutions explicitly, we omit them here as they are not needed for our purposes.

	\begin{lemma} {\rm \cite[Lemma 3.1]{Bo_Bo_22}} \label{lem Bo}
Let \( A \in \mathrm{M}_n(D) \) and \( B \in \mathrm{M}_m(D) \), where \( n \) and \( m \) are positive integers. Suppose there exists a polynomial \( p \) in one variable with coefficients in the center of \( D \) such that \( p(A) = 0 \) and \( p(B) \) is invertible. Then, for any \( C \in \mathrm{M}_{n \times m}(D) \), the Sylvester equation  
\[
AX - XB = C
\]
has a unique solution \( X \in \mathrm{M}_{n \times m}(D) \).
	\end{lemma}
	
	Lemma~\ref{lem Bo} is a specialized case of a more general result established by P. M. Cohn in \cite[Lemma 2.3]{Pa_Co_73} for algebras over a common field. Given our focus on matrices, we prefer to cite the recent result of Bolotnikov.
	
	With Lemma~\ref{lem Bo} in place, we now proceed to appropriately select the matrices \( B \) and \( a \). For our purpose, we now combine the previous observation with Lemma~\ref{lem Bo}, stating it as follows.

	\begin{lemma}\label{cheo khoi}
		Let \( B \in \mathrm{M}_n(D), \alpha\in\mathrm{M}_{n \times 1}(D) \),  and \( a \in D \).	If there exists a polynomial \( p(t) \) with coefficients in the center of \( D \) such that  such that \( p(B) = 0 \) and  \( p(a) \) is invertible, then the matrix $\begin{pmatrix}
			B & \alpha \\
			0 & a
		\end{pmatrix}$ is similar to the matrix $\begin{pmatrix}
			B & 0 \\
			0 & a
		\end{pmatrix}.$
	\end{lemma}
	
	For the time being, we will set this aside and return to it later.   Now, let us move on to the next step in the proof. As mentioned earlier, we shall use the rational canonical form for matrices over \( D \) as described in \cite[Chapter 3, Section 12, Page 50]{Bo_Ja_43} and \cite[Section 8.4, Page 505]{Bo_Co_85}, which is outlined below.
	
	\begin{lemma}\label{companion}
		Every square matrix over $D$ is similar to a block diagonal matrix  of companion matrices over $D$.	
	\end{lemma}
	
	With Lemma~\ref{companion} and Remark~\ref{remark di} in mind, we now turn our attention to decomposing companion matrices into sums of diagonalizable matrices. Motivated by the proof of \cite[Lemma 2.1]{Pa_Bo_20}, we begin with a series of observations. Note that $D$ is always assumed to be a noncommutative division ring, unless stated otherwise.
	
	For $x\in D\setminus\{0\}$, let
	$$G_n(x)=\begin{cases}
		\displaystyle	\bigoplus_{\frac{n}{2}}\begin{pmatrix}
			0&0\\1&-x
		\end{pmatrix} \text{ if } n \text{ is even},\\ \displaystyle
		\left(\bigoplus_{\frac{n-1}{2}}\begin{pmatrix}
			0&0\\1&-x
		\end{pmatrix}\right)\oplus(0) \text{ if } n\geq3 \text{  is odd},
	\end{cases}$$ and $$H_n(x)=\begin{cases}
		\displaystyle (0)\oplus\left(\bigoplus_{\frac{n-2}{2}}\begin{pmatrix}
			x&0\\1&0
		\end{pmatrix}\right)\oplus(x) \text{ if } n \text{ is even},\\ \displaystyle (0)\oplus\left(\bigoplus_{\frac{n-1}{2}}\begin{pmatrix}
			x&0\\1&0
		\end{pmatrix}\right) \text{ if } n\geq3 \text{  is odd}.
	\end{cases}$$ As we know, we will decompose companion matrices into sums of diagonalizable matrices. Let $$A=\begin{pmatrix}
		0& & &a_0 \\
		1& & &a_1\\
		&\ddots& &\vdots\\
		0& &1&a_{n-1}
	\end{pmatrix}\in \mathrm{M}_n(D),$$ where $a_0,a_1,\ldots,a_{n-1}\in D$. The case where \( n = 1 \) is trivial, so we restrict ourselves to the case where \( n \geq 2 \). If $n=2$, then \begin{eqnarray*}
		A&=&\begin{pmatrix}
			0&a_0\\
			1&a_1
		\end{pmatrix}=\begin{pmatrix}
			0&0\\
			1&a_1-d
		\end{pmatrix}+\begin{pmatrix}
			0&a_0\\
			0&d
		\end{pmatrix}\\
		&=&\begin{pmatrix}
			1&0\\ (d-a_1)^{-1}&1
		\end{pmatrix} \begin{pmatrix}
			0&0\\0&a_1-d
		\end{pmatrix} \begin{pmatrix}
			1&0\\ (d-a_1)^{-1}&1
		\end{pmatrix}^{-1}\\
		&&+\begin{pmatrix}
			1&a_0d^{-1}\\
			0&1
		\end{pmatrix} \begin{pmatrix}
			0&0\\0&d
		\end{pmatrix}\begin{pmatrix}
			1&a_0d^{-1}\\
			0&1
		\end{pmatrix}^{-1}
	\end{eqnarray*} is a sum of two diagonalizable matrices over $D$, where $d\in D\setminus\{a_1,0\}$, which completes the case where $n=2$. Note that the element \( d \) can certainly be chosen. In fact, recalling that \( D \) is a noncommutative division ring.  A celebrated theorem of Wedderburn, as stated in \cite[(13.1) Wedderburn's ``Little" Theorem, p. 203]{Bo_La_91}, tells us that every finite division ring is a field. Hence, $D$ must be infinite, and selecting such an element \( d \) is possible. By the way, since this theorem is frequently used throughout the paper, we state it here for convenience. Although it may not perfectly align with the content of this subsection, its placement here is logically appropriate.
	
	\begin{theorem} \label{Wedderburn}
		Every finite division ring is a field.
	\end{theorem}
	
	We now turn our attention to the case where \( n \geq 3 \). In this case, our proof follows from certain ideas inspired by Botha, through the following four decompositions, which distinguish between the cases where \( n \) is even and odd, and between the cases where \( a_{n-1} \) is zero and nonzero. First, note that \({}^T\) is used to denote the transpose operation, and let $$v = \begin{pmatrix}
		a_0&a_1&\cdots&a_{n-2}
	\end{pmatrix}^T,u=\begin{pmatrix}
		0&0&\cdots&0&1
	\end{pmatrix}^T \in \mathrm{M}_{(n-1) \times 1}(D).$$ Later, we will consider the following four cases, under the assumption that the center of \(D\) contains at least three elements.
	
	\medskip
	
	\noindent\textbf{Case 1.} $n\geq3$ is odd and $a_{n-1}=0$. Then, $$A=\begin{pmatrix}
		G_{n-1}(x)&v\\
		0&-1
	\end{pmatrix}+\begin{pmatrix}
		H_{n-1}(x)&0\\
		u&1
	\end{pmatrix}.$$ Note that the second summand \(\begin{pmatrix} H_{n-1}(x) & 0 \\ u & 1 \end{pmatrix}\) is exactly the block diagonal matrix
	\[
	(0) \oplus \left( \bigoplus_{\frac{n-3}{2}} \begin{pmatrix} x & 0 \\ 1 & 0 \end{pmatrix} \right) \oplus \begin{pmatrix} x & 0 \\ 1 & 1 \end{pmatrix}.
	\] Moreover, it is easy to verify that the following matrices are diagonalizable:\[
	\begin{pmatrix}x & 0 \\ 1 & 0\end{pmatrix} = \begin{pmatrix}0 & x \\ 1 & 1\end{pmatrix} \begin{pmatrix} 0 & 0 \\ 0 & x \end{pmatrix} \begin{pmatrix}0 & x \\ 1 & 1\end{pmatrix}^{-1},
	\] and
	\[
	\begin{pmatrix} x & 0 \\ 1 & 1 \end{pmatrix} = \begin{pmatrix} 0 & x - 1 \\ 1 & 1 \end{pmatrix} \begin{pmatrix} 1 & 0 \\ 0 & x \end{pmatrix} \begin{pmatrix} 0 & x - 1 \\ 1 & 1 \end{pmatrix}^{-1},
	\] provided that \(x \notin \{0, 1\}\). Hence, taking Remark~\ref{remark di} into account, the second summand \(\begin{pmatrix} H_{n-1}(x) & 0 \\ u & 1 \end{pmatrix}\) is diagonalizable. 
	
	Returning to the first summand \(\begin{pmatrix} G_{n-1}(x) & v \\ 0 & -1 \end{pmatrix}\), we will use Lemma~\ref{cheo khoi} to perform the diagonalization. If \(x\) belongs to the center of \(D\), then the polynomial \(p(z) = z^2 + xz\) in the variable $z$ vanishes on \(\begin{pmatrix} 0 & 0 \\ 1 & -x \end{pmatrix}\). Consequently, \(p(z)\) vanishes on \(G_{n-1}(x)\). Moreover, if \(x \neq 1\), then \(p(-1)\) is invertible. Therefore, by using Lemma~\ref{cheo khoi}, the first summand \(\begin{pmatrix} G_{n-1}(x) & v \\ 0 & -1 \end{pmatrix}\) is similar to the block diagonal matrix \(G_{n-1}(x) \oplus (-1)\). Taking Remark~\ref{remark di} into account, it suffices to check whether the matrix \(G_{n-1}(x) \oplus (-1)\) is diagonalizable. It is straightforward to show that $$\begin{pmatrix} 0 & 0 \\ 1 & -x \end{pmatrix} = \begin{pmatrix} x & 0 \\ 1 & 1 \end{pmatrix} \begin{pmatrix} 0 & 0 \\ 0 & -x \end{pmatrix} \begin{pmatrix} x & 0 \\ 1 & 1 \end{pmatrix}^{-1}$$ is diagonalizable, and thus, by Remark~\ref{remark di}, the matrix \(G_{n-1}(x) \oplus (-1)\) is diagonalizable as well.
	
	To summarize, in this case, we must choose \(x \in D\) such that \(x\) belongs to the center of \(D\) and \(x \notin \{0, 1\}\). This is possible due to the assumption that the center of \(D\) contains at least three elements.
	
	\medskip
	
	\noindent\textbf{Case 2.} $n\geq3$ is odd and $a_{n-1}\neq0$. Following a similar argument as in Case 1, with a slight modification, we have the expression 
	\[
	A = \begin{pmatrix} G_{n-1}(x) & v \\ 0 & a_{n-1} \end{pmatrix} + \begin{pmatrix} H_{n-1}(x) & 0 \\ u & 0 \end{pmatrix}.
	\]
	Note that the second summand, \(\begin{pmatrix} H_{n-1}(x) & 0 \\ u & 0 \end{pmatrix}\), is precisely the block diagonal matrix
	\[
	(0) \oplus \left( \bigoplus_{\frac{n-1}{2}} \begin{pmatrix} x & 0 \\ 1 & 0 \end{pmatrix} \right).
	\]
	It follows from Case 1 that if \(x \neq 0\), the matrix \(\begin{pmatrix} x & 0 \\ 1 & 0 \end{pmatrix}\) is diagonalizable; hence, by Remark~\ref{remark di}, the second summand \(\begin{pmatrix} H_{n-1}(x) & 0 \\ u & 0 \end{pmatrix}\) is also diagonalizable.
	
	As discussed in Case 1, to handle the first summand, we must assume that \(x \neq -a_{n-1}\) and that \(x\) belongs to the center of \(D\). This assumption is valid due to the fact that the center of \(D\) contains at least three elements. With these conditions in place, we can conclude that the first summand, \(\begin{pmatrix} G_{n-1}(x) & v \\ 0 & a_{n-1} \end{pmatrix}\), is diagonalizable, as desired.
	
	For clarity, we emphasize that in this case, we choose \(x \in D\) such that \(x\) belongs to the center of \(D\) and \(x \notin \{0, -a_{n-1}\}\), which is possible by the assumption that the center of \(D\) contains at least three elements.
	
	\medskip
	
	\noindent\textbf{Case 3.} $n\geq4$ is even and $a_{n-1}=0$. In this case, the expression for \(A\) is only slightly different from those in Case 1 and Case 2, though the underlying arguments remain largely the same, with only a minor modification.  We now present the expression for \(A\), which is given by: 
	$$A = \begin{pmatrix} 
		G_{n-1}(x) & 0 \\
		u & -1
	\end{pmatrix} + \begin{pmatrix}
		H_{n-1}(x) & v \\
		0 & 1
	\end{pmatrix}.$$
	
	Note that the first summand, \(\begin{pmatrix} G_{n-1}(x) & 0 \\ u & -1 \end{pmatrix}\), is precisely the block diagonal matrix 
	$$\left( \bigoplus_{\frac{n-2}{2}} \begin{pmatrix} 0 & 0 \\ 1 & -x \end{pmatrix} \right) \oplus \begin{pmatrix} 0 & 0 \\ 1 & -1 \end{pmatrix}.$$
	As discussed in Case 1, both \(\begin{pmatrix} 0 & 0 \\ 1 & -x \end{pmatrix}\) and \(\begin{pmatrix} 0 & 0 \\ 1 & -1 \end{pmatrix}\) are diagonalizable if \(x \neq 0\). Hence, by Remark~\ref{remark di}, the first summand \(\begin{pmatrix} G_{n-1}(x) & 0 \\ u & -1 \end{pmatrix}\) is diagonalizable, as expected. 
	
	For the second summand, we will proceed as follows. Since the polynomial \(p(z) = z^2 - xz\) in the variable $z$ vanishes on \(H_{n-1}(x)\), and \(p(1)\) is invertible under the assumption that \(x \neq 1\) and \(x\) belongs to the center of \(D\), it follows from Lemma~\ref{cheo khoi} that the second summand \(\begin{pmatrix} H_{n-1}(x) & v \\ 0 & 1 \end{pmatrix}\) is similar to the block diagonal matrix 
	$$ (0) \oplus \left( \bigoplus_{\frac{n-2}{2}} \begin{pmatrix} x & 0 \\ 1 & 0 \end{pmatrix} \right) \oplus (1).$$
	As noted in Case 1, the matrix \(\begin{pmatrix} x & 0 \\ 1 & 0 \end{pmatrix}\) is diagonalizable. Therefore, by Remark~\ref{remark di}, the matrix 
	$$ (0) \oplus \left( \bigoplus_{\frac{n-2}{2}} \begin{pmatrix} x & 0 \\ 1 & 0 \end{pmatrix} \right) \oplus (1) $$ 
	is diagonalizable. Using Remark~\ref{remark di} once again, we conclude that the second summand \(\begin{pmatrix} H_{n-1}(x) & v \\ 0 & 1 \end{pmatrix}\) is also diagonalizable.
	
	Finally, we emphasize that in this case, we require \(x \in D\) such that \(x\) belongs to the center of \(D\) and \(x \notin \{0, 1\}\). This condition is possible due to the standing assumption that the center of \(D\) contains at least three elements.
	
	\medskip

	\noindent\textbf{Case 4.} $n\geq4$ is even and $a_{n-1}\neq0$. We begin by expressing \( A \) as follows:
	\[
	A = \begin{pmatrix} 
		G_{n-1}(x) & 0 \\
		u & -x 
	\end{pmatrix} + \begin{pmatrix}
		H_{n-1}(x) & v \\
		0 & a_{n-1} + x
	\end{pmatrix}.
	\]
	Note that the first summand, \( \begin{pmatrix} G_{n-1}(x) & 0 \\ u & -x \end{pmatrix} \), is explicitly the block diagonal matrix:
	\[
	\bigoplus_{\frac{n}{2}} \begin{pmatrix} 0 & 0 \\ 1 & -x \end{pmatrix}.
	\]
	As discussed in Case 1, the matrix \( \begin{pmatrix} 0 & 0 \\ 1 & -x \end{pmatrix} \) is diagonalizable when \( x \neq 0 \). Hence, by Remark~\ref{remark di}, the first summand \( \begin{pmatrix} G_{n-1}(x) & 0 \\ u & -x \end{pmatrix} \) is also diagonalizable.
	
	For the second summand, we require that \( x \neq -a_{n-1} \) and that \( x \) belongs to the center of \( D \). This assumption holds, as the center of \( D \) is assumed to contain at least three elements. As in Case 3, it is known that the polynomial \( p(z) = z^2 - xz \) in the variable $z$ vanishes on \( H_{n-1}(x) \), and that \( p(a_{n-1} + x) \) is invertible. Therefore, by Lemma~\ref{cheo khoi}, the second summand \( \begin{pmatrix} H_{n-1}(x) & v \\ 0 & a_{n-1}+x \end{pmatrix} \) is similar to the block diagonal matrix
	\[
	(0) \oplus \left( \bigoplus_{\frac{n-2}{2}} \begin{pmatrix} x & 0 \\ 1 & 0 \end{pmatrix} \right) \oplus (a_{n-1} + x).
	\]
	At this point, it is clear that the matrix \( \begin{pmatrix} x & 0 \\ 1 & 0 \end{pmatrix} \) is diagonalizable if $x\neq0$. Using Remark~\ref{lem Bo} twice consecutively, we conclude that the second summand $$ \begin{pmatrix} H_{n-1}(x) & v \\ 0 & 1 \end{pmatrix}$$ is diagonalizable, as anticipated.
	
	Finally, to reiterate, in this case, we must choose \( x \in D \) such that \( x \) belongs to the center of \( D \) and \( x \notin \{0, -a_{n-1}\} \), which is consistent with the assumption that the center of \( D \) contains at least three elements. Moreover, it is obvious that the number $2$ is sharp.

	Following the cases we have discussed and proven earlier, along with the result of Botha in \cite[Theorem 2.6]{Pa_Bo_20}, we now state the following result, which serves as the first main result of this paper.
	
	\begin{theorem}\label{sum1}
		Let $D$ be a division ring and let $n$ be a positive integer. If the center of $D$ contains at least three elements, then every matrix in $\mathrm{M}_n(D)$ can be expressed as a sum of two diagonalizable matrices in $\mathrm{M}_n(D)$. Moreover, the number $2$ is sharp.
	\end{theorem}

	Let \( K \) be a field. It is shown in either \cite[\S4, Example 3]{Pa_ChuLee_79} or \cite[Proposition~2.3.5]{Bo_Co_95} that noncommutative division rings with center \( K \) exist. In particular, if \( K \) is the field of two elements, then there always exist noncommutative division rings with center \( K \). Furthermore, such division rings must be infinite-dimensional vector spaces over their centers, and each element commutes with at least one element that is not algebraic over the center. This follows directly from \cite[(13.11) Theorem (Jacobson)]{Bo_La_91}, which states that any division ring that is algebraic over its finite center must be commutative. Recall a division ring $D$ is said to be \textit{algebraic} if every element in $D$ is a root of a polynomial in one variable with coefficients from the center of $D$. Moreover, it is not difficult to verify that any division ring of finite-dimensional over its center must be an algebraic division ring.

	\subsection{Division rings with center containing only two elements}
	
The case where \( D \) is the field of two elements has been thoroughly addressed. In particular, Botha established in \cite[Theorem 2.6]{Pa_Bo_20} that every matrix over \( D \) can be expressed as a sum of three diagonalizable matrices. Moreover, a matrix over \( D \) can be expressed as a sum of two diagonalizable matrices if and only if it is similar to a matrix of the form  
	\[ 
	N \oplus \begin{pmatrix} 
		0 & X \\ 
		\mathrm{I} & 0 
	\end{pmatrix} \oplus (\mathrm{I} + M), 
	\]  
	where \( N \) and \( M \) are nilpotent, \( X \) and \( X + \mathrm{I} \) are nonsingular, and \( \mathrm{I} \) represents the identity matrix of suitable size. Additionally, the number $3$ is sharp.

	In this subsection, we show that if the center of \( D \) is the field of two elements, every matrix over \( D \) can be written as the sum of three diagonalizable matrices over $D$. Directly benefiting from Botha's result \cite[Theorem 2.6]{Pa_Bo_20}, we see that  the number $3$ in this case is sharp. To establish this, we revisit the proof of Theorem~{\rm \ref{sum1}}, employing the matrices \( G_n(x) \) and \( H_n(x) \), as well as the row \( u \) and column \( v \).
	The tools are now in place to validate the following theorem:
	
	\begin{theorem}\label{sum2}
		Let $D$ be a division ring and let $n$ be a positive integer. If the center of $D$ is the field of two elements, then every matrix over \( D \) can be written as a sum of three diagonalizable matrices over $D$. Moreover, the number $3$ is sharp.
	\end{theorem}
	
	\begin{proof}
		We know from the result of Botha \cite[Theorem 2.6]{Pa_Bo_20} that the proof is completed in the case where $D$ is commutative. Hence, we focus on the case where $D$ is noncommutative. By Theorem~\ref{Wedderburn}, every finite division ring is a field, implying that \( D \) must be infinite. 
		
		By Lemma~\ref{companion} and Remark~\ref{remark di}, it suffices to establish the result for companion matrices. Let \( A \in \mathrm{M}_n(D) \) be a companion matrix, and consider the matrices \( G_n(x) \), \( H_n(x) \), as well as the row \( u \) and column \( v \), as  in the proof of Theorem~\ref{sum1}. It is clear that the case \( n=1 \) is trivial, while the case \( n=2 \), already addressed in the proof of Theorem~\ref{sum1}, remains valid. For the sake of clarity, we provide a brief illustration of the case \( n=2 \) below.
		
		If \( n=2 \), we have 
		\[
		A = 
		\begin{pmatrix}
			0 & a_0 \\
			1 & a_1
		\end{pmatrix} 
		= 
		\begin{pmatrix}
			0 & 0 \\
			1 & a_1 - d
		\end{pmatrix} 
		+ 
		\begin{pmatrix}
			0 & a_0 \\
			0 & d
		\end{pmatrix},
		\]
		which can be rewritten as 
		\begin{eqnarray*}
			A &=& 
			\begin{pmatrix}
				1 & 0 \\
				(d - a_1)^{-1} & 1
			\end{pmatrix} 
			\begin{pmatrix}
				0 & 0 \\
				0 & a_1 - d
			\end{pmatrix} 
			\begin{pmatrix}
				1 & 0 \\
				(d - a_1)^{-1} & 1
			\end{pmatrix}^{-1} \\
			&&+ 
			\begin{pmatrix}
				1 & a_0 d^{-1} \\
				0 & 1
			\end{pmatrix} 
			\begin{pmatrix}
				0 & 0 \\
				0 & d
			\end{pmatrix} 
			\begin{pmatrix}
				1 & a_0 d^{-1} \\
				0 & 1
			\end{pmatrix}^{-1},
		\end{eqnarray*}
		where \( d \in D \setminus \{a_1, 0\} \). This decomposition expresses \( A \) as a sum of two diagonalizable matrices over \( D \). Notably, the choice of \( d \) is feasible even when the center of \( D \) does not contain at least three elements, but $D$ must be infinite.
		
		We now proceed to the case \( n \geq 3 \). The proof is divided into two cases and follows a similar strategy to the proof by Botha in \cite[Lemma 2.3]{Pa_Bo_20}. However, we present it here for completeness.
		
		\medskip
		
		\noindent\textbf{Case 1.} \( n \geq 3 \) is odd.  
		In this case, we use the decomposition:
		\[
		A = 
		\begin{pmatrix}
			G_{n-1}(x) & 0 \\
			0 & 1 + a_{n-1}
		\end{pmatrix} 
		+ 
		\begin{pmatrix}
			H_{n-1}(x) & 0 \\
			u & 0
		\end{pmatrix} 
		+ 
		\begin{pmatrix}
			0 & v \\
			0 & 1
		\end{pmatrix},
		\]
		where \( x = -x = 1 \). Since \( x = 1 \neq 0 \), the first term 
		\[
		\begin{pmatrix}
			G_{n-1}(x) & 0 \\
			0 & 1 + a_{n-1}
		\end{pmatrix}
		\] 
		is diagonalizable by Lemma~\ref{cheo khoi}. Similarly, the second term 
		\[
		\begin{pmatrix}
			H_{n-1}(x) & 0 \\
			u & 0
		\end{pmatrix}
		\] 
		and the third term 
		\[
		\begin{pmatrix}
			0 & v \\
			0 & 1
		\end{pmatrix}
		\] 
		are, respectively, similar to \( H_{n-1}(x) \oplus (0) \) and \( \left( \bigoplus_{n-1}(0) \right) \oplus (1) \), both of which are diagonalizable. By Remark~\ref{remark di}, the proof for this case is complete.
		
		\medskip
		
		\noindent\textbf{Case 2.}  \( n \) is even.  
		For this case, consider the decomposition:
		\[
		A = 
		\begin{pmatrix}
			G_{n-1}(x) & 0 \\
			u & 1
		\end{pmatrix} 
		+ 
		\begin{pmatrix}
			H_{n-1}(x) & 0 \\
			0 & a_{n-1}
		\end{pmatrix} 
		+ 
		\begin{pmatrix}
			0 & v \\
			0 & 1
		\end{pmatrix},
		\]
		where \( x = -x = 1 \). By following the same reasoning as in the previous case, we conclude that all three summands are diagonalizable, completing the proof.

		The claim that the number \( 3 \) is sharp is confirmed by \cite[Lemma 2.2]{Pa_Bo_98}. In particular, the matrix 
		\[
		\begin{pmatrix}
			1 & 1 \\
			1 & 0
		\end{pmatrix} 
		\oplus \mathrm{I}_{n-2}
		\] 
		serves as a counterexample to any decomposition into fewer than three diagonalizable matrices.
	\end{proof}
	
	In \cite[Lemma 2.5]{Pa_Bo_20}, Botha established the necessary and sufficient conditions for a matrix over the field of two elements to be expressed as the sum of two diagonalizable matrices. While this result is quite intriguing, we are currently unaware of a method to approach the problem in the more general context of noncommutative division rings. Consequently, we leave this as an open problem.
	
	\begin{open}\label{open1}
		Let \( D \) be a division ring such that the center of \( D \) is the field of two elements. Find a necessary and sufficient condition for a matrix over \( D \) to be the sum of two diagonalizable matrices.
	\end{open}

	Even without the condition that the center of \( D \) contains at least three elements, some matrices, such as nilpotent matrices, can still be expressed as the sum of two diagonalizable matrices. To illustrate this, let \( N \in \mathrm{M}_n(D) \) be a nilpotent matrix. As shown in \cite[Theorem~5]{Pa_Mo_12}, every nilpotent matrix over a division ring is similar to a block diagonal matrix consisting of Jordan blocks, where all entries on the main diagonal are \( 0 \). In particular, there exists \( P \in \mathrm{GL}_n(D) \) such that  
	\[
	P^{-1} N P = \bigoplus_{i=1}^s J_{m_i}(0),  
	\]  
	for some positive integers \( m_1, m_2, \ldots, m_s \), \( s \), where \( m_1 + m_2 + \cdots + m_s = n \). Here,  
	\[
	J_{m_i}(0) = \begin{pmatrix} 
		0 & 1 & 0 & \cdots & 0 & 0 & 0 \\ 
		0 & 0 & 1 & \cdots & 0 & 0 & 0 \\ 
		0 & 0 & 0 & \cdots & 0 & 0 & 0 \\ 
		\vdots & \vdots & \vdots & \ddots & \vdots & \vdots & \vdots \\ 
		0 & 0 & 0 & \cdots & 0 & 1 & 0 \\ 
		0 & 0 & 0 & \cdots & 0 & 0 & 1 \\ 
		0 & 0 & 0 & \cdots & 0 & 0 & 0 
	\end{pmatrix} \in \mathrm{M}_{m_i}(D).
	\]  
	By Remark~\ref{remark di}, it suffices to show that each \( J_{m_i}(0) \) can be expressed as the sum of two diagonalizable matrices over \( D \). Even when the center of \( D \) is the field of two elements, we can always select \( x \in D \) such that both \( G_{m_i}(x) \) and \( H_{m_i}(x) \) are diagonalizable. Furthermore, it is straightforward to verify that $$ J_{m_i}(0) = G_{m_i}(x) + H_{m_i}(x).$$ Thus, \( J_{m_i}(0) \) can indeed be expressed as the sum of two diagonalizable matrices, as claimed.  
	
	For convenience, we summarize this result as follows, which  offers insight into Problem~\ref{open1}.

	\begin{theorem}\label{sum3}
		Let $D$ be a division ring and let $n$ be a positive integer. If a matrix $N\in\mathrm{M}_n(D)$ is nilpotent, then the matrix $N$ can be expressed as a sum of two diagonalizable matrices over $D$. Moreover, the number $2$ is sharp.
	\end{theorem}
	
Following the proof of Theorem~\ref{sum2}, we obtain the following result, which contributes to addressing Problem~\ref{open1}.

\begin{theorem}
If $D$ is a noncommutative division ring, then every matrix in $\mathrm{M}_2(D)$ can be expressed as a sum of two diagonalizable matrices in $\mathrm{M}_2(D)$. Moreover, the number $2$ is sharp.
\end{theorem}

	\section{Products of diagonalizable matrices}\label{product}

	Between 1998 and 1999, Botha investigated in \cite{Pa_Bo_98,Pa_Bo_99} the decomposition of matrices over fields into products of diagonalizable matrices, arriving at a noteworthy result: any matrix over a field with at least four elements can be expressed as a product of two diagonalizable matrices. In this section, we continue this result by considering noncommutative division rings in place of fields. 
	
	Building on the previous section, we observe that, with the help of Lemma~\ref{companion} and Remark~\ref{remark di}, it is sufficient to focus on companion matrices. To proceed, we divide the discussion into three cases, which are presented in the following sections. Throughout this section, let $$A=\begin{pmatrix}
		0& & &a_0 \\
		1& & &a_1\\
		&\ddots& &\vdots\\
		0& &1&a_{n-1}
	\end{pmatrix}\in \mathrm{M}_n(D),$$ where $a_0,a_1,\ldots,a_{n-1}\in D$. The case where \( n = 1 \) is straightforward, so we focus on the case where \( n \geq 2 \).

	\subsection{Division rings of characteristic different from 2}
	
	Building upon Botha's approach in \cite[Theorem 2.1]{Pa_Bo_98}, we start with the following factorization of the matrix:  
	\[
	A = \begin{pmatrix}
		&&1\\
		&\iddots&\\
		1&&
	\end{pmatrix} 
	\begin{pmatrix}
		& &1&a_{n-1}\\
		&\iddots& &\vdots\\
		1& & &a_1\\
		0&\ldots &0 &a_0
	\end{pmatrix}.
	\]  
	It is straightforward to verify that  
	\[
	\begin{pmatrix}
		&&1\\
		&\iddots&\\
		1&&
	\end{pmatrix}^2 = \mathrm{I}_n,
	\]  
	indicating that this matrix is an involution in \(\mathrm{M}_n(D)\). As shown in \cite[Proposition~2.3]{Pa_BiDuHaSo_23}, every involution in \(\mathrm{M}_n(D)\) is diagonalizable if the characteristic of \(D\) is different from \(2\). Consequently, the matrix  
	\[
	\begin{pmatrix}
		&&1\\
		&\iddots&\\
		1&&
	\end{pmatrix}
	\]  
	is diagonalizable.  
	
	If \( a_0 \neq \pm1 \), we claim that the matrix  
	\[
	\begin{pmatrix}
		& &1&a_{n-1}\\
		&\iddots& &\vdots\\
		1& & &a_1\\
		0&\ldots&0 &a_0
	\end{pmatrix}
	\]  
	is similar to the block diagonal matrix  
	\[
	\begin{pmatrix}
		&&1\\
		&\iddots&\\
		1&&
	\end{pmatrix} \oplus (a_0).
	\]  
	Indeed, the polynomial \( p(z) = z^2 -1 \) in the variable $z$ vanishes on any involution in \(\mathrm{M}_n(D)\), and since \( a_0 \neq \pm1 \), it follows that \( p(a_0) \) is invertible. By Lemma~\ref{cheo khoi}, the matrix is diagonalizable. This confirms the result for the case \( a_0 \neq \pm1 \).  
	
	Next, consider the case \( a_0 = \pm1 \). For the case where $D$ is commutative, the result is already established by Botha in \cite[Theorem 2.1]{Pa_Bo_98}. We therefore examine the case where \( D \) is noncommutative. By Theorem~\ref{Wedderburn}, every finite division ring is a field, implying that \( D \) must be infinite. Thus, we can select \( a \in D \) such that \( a^2 \neq \pm a_0 \).  
	
	With this choice, consider the following factorization:  
	\[
	A = \begin{pmatrix}
		& & &a^2\\
		& &1& \\
		&\iddots& &\\
		1& & &
	\end{pmatrix} 
	\begin{pmatrix}
		& &1&a_{n-1}\\
		&\iddots&&\vdots\\
		1&&&a_1\\
		0&\ldots&0&a^{-2}a_0
	\end{pmatrix}.
	\]  
	Since \( a^2 \neq \pm a_0 \), it follows that \( a^{-2}a_0 \neq \pm1 \). By Lemma~\ref{cheo khoi}, the matrix  
	\[
	\begin{pmatrix}
		& &1&a_{n-1}\\
		&\iddots&&\vdots\\
		1&&&a_1\\
		0&\ldots&0&a^{-2}a_0
	\end{pmatrix}
	\]  
	is similar to the block diagonal matrix  
	\[
	\begin{pmatrix}
		&&1\\
		&\iddots&\\
		1&&
	\end{pmatrix} \oplus (a^{-2}a_0).
	\]  
	Since the matrix  
	\[
	\begin{pmatrix}
		&&1\\
		&\iddots&\\
		1&&
	\end{pmatrix}
	\]  
	is diagonalizable, it follows from Remark~\ref{remark di} that the matrix  
	\[
	\begin{pmatrix}
		& &1&a_{n-1}\\
		&\iddots&&\vdots\\
		1&&&a_1\\
		0&\ldots&0&a^{-2}a_0
	\end{pmatrix}
	\]  
	is also diagonalizable.  Finally, we address the diagonalization of the matrix  
	\[
	\begin{pmatrix}
		& & &a^2\\
		& &1& \\
		&\iddots& &\\
		1& & &
	\end{pmatrix}.
	\]  
	Note that the entries of this matrix do not necessarily belong to the center of \(D\) due to the presence of \(a^2\). However, we claim that the entries of this matrix belong to a field contained in \(D\). Let \(K\) be the subdivision ring of \(D\) generated by \(a^2\) over the center of \(D\). It is straightforward to verify that \(K\) is indeed a field containing the center of \(D\). Therefore, it follows that  
	\[
	\begin{pmatrix}
		& & &a^2\\
		& &1& \\
		&\iddots& &\\
		1& & &
	\end{pmatrix} \in \mathrm{M}_n(K).
	\]  
	Given the standing assumption that \(D\) has characteristic different from \(2\), it follows that \(K\) also has characteristic different from \(2\). Consequently, by \cite[Lemma 1.1(b)]{Pa_Bo_98}, we conclude that the matrix  
	\[
	\begin{pmatrix}
		& & &a^2\\
		& &1& \\
		&\iddots& &\\
		1& & &
	\end{pmatrix}
	\]  
	is diagonalizable.

	With these observations, we confirm the validity of the following result, which constitutes one of the main contributions of this paper.

	\begin{theorem}\label{cha khac 2}
		Let $D$ be a division ring of characteristic different from $2$ and let $n$ be a positive integer. If  $D$ is not the field of three elements, then every matrix in $\mathrm{M}_n(D)$ is expressible as a product of two diagonalizable matrices in $\mathrm{M}_n(D)$. Moreover, the number $2$ is sharp.
	\end{theorem}
	
	In the remaining case where $D$ is the field of three elements, it is shown in \cite[Theorem 2.4]{Pa_Bo_98} that every matrix in $\mathrm{M}_n(D)$ is expressible as a product of three diagonalizable matrices in $\mathrm{M}_n(D)$. Moreover, the number $3$ is sharp.

	\subsection{Division rings of characteristic 2}
	
	Botha showed in \cite[Theorem 2.2]{Pa_Bo_99} that if $D$ is a field of characteristic $2$, containing at least four elements, then every matrix over $D$ can be written as a product of two diagonalizable matrices over $D$. To the best of our knowledge, there is currently no known technique to address the case where $D$ is noncommutative and the characteristic of \( D \) is $2$. While Botha addressed  the case in \cite{Pa_Bo_99} where \( D \) is a field of characteristic $2$, it is important to note that results for division rings often differ significantly from those for fields. In many instances, satisfactory analogues are lacking, and the existing results remain incomplete. 
	
	Building on the previous subsection, we know that any involution in $\mathrm{M}_n(D)$ is diagonalizable when $D$ has a characteristic different from 2. However, as established in \cite[Proposition 2.3]{Pa_BiDuHaSo_23}, if $D$ has characteristic 2, then every involution in $\mathrm{M}_n(D)$ is similar to a matrix over the center of $D$. Consequently, following the proof of Theorem~\ref{cha khac 2} and incorporating the result from Botha in \cite[Theorem 2.2]{Pa_Bo_99}, we deduce that if $D$ has characteristic 2 and contains at least four elements, every matrix in $\mathrm{M}_n(D)$ can be written as the product of four diagonalizable matrices over $D$. For clarity, we summarize this result as follows:
	
	\begin{theorem}\label{cha 2}
		Let $D$ be a noncommutative division ring of characteristic  $2$ and let $n$ be a positive integer. If  $D$ contains at least four elements, then every matrix in $\mathrm{M}_n(D)$ is expressible as a product of four diagonalizable matrices in $\mathrm{M}_n(D)$.
	\end{theorem}
	
	The number $4$ is not sharp. Therefore, we conjecture that the number $2$ is sharp, as this bound is supported by Botha in \cite[Theorem 2.2]{Pa_Bo_99}. We thus leave the following open problem:
	
	\begin{open} \label{open2}
		Can the number $4$ be replaced by $2$ in Theorem~{\rm \ref{cha 2}}? 
	\end{open}
	
	In an attempt to address Problem~\ref{open2}, we will argue that the answer is ``yes" in the subsequent subsection, provided that the center of \( D \) contains sufficiently many elements, regardless of the characteristic of $D$.
	
	Similarly, at the end of Section~\ref{sum}, and by combining the results from \cite[Corollary~2.2]{Pa_Bo_98} and \cite[Theorem 2.2]{Pa_Bo_99}, we can easily verify the validity of the following theorem, which provides an example to consider for Problem~\ref{open2}.
	
	\begin{theorem}\label{nil2}
		Let $D$ be a division ring such that $D$ is not the field of two elements and let $n$ be a positive integer. If a matrix $N\in\mathrm{M}_n(D)$ is nilpotent, then the matrix $N$ can be expressed as a product of two diagonalizable matrices over $D$. Moreover, the number $2$ is sharp.
	\end{theorem}
	
	Furthermore, in the case where \( D = F \) is the field with two elements, \cite[Corollary~3.2]{Pa_Bo_98} establishes that the only nonsingular matrix in \( \mathrm{M}_n(D) \) that can be expressed as a product of diagonalizable matrices in \( \mathrm{M}_n(D) \) is \( \mathrm{I}_n \). Interestingly, despite this restrictive result, \cite[Theorem 4.6]{Pa_Bo_98} shows that every singular matrix over \( D \) can still be expressed as a product of \( t \) diagonalizable matrices, where \( t \) depends on some conditions.

	\subsection{Division rings with center containing sufficiently many elements}

	Let \( F \) denote the center of \( D \), where \( F \) contains a sufficiently large number of elements. Motivated by \cite[Theorem 3.3]{Pa_Bo_98}, consider the following: for $$ b_0, b_1, \ldots, b_{n-1}, c_0, c_1, \ldots, c_{n-1} \in D,$$ it is straightforward to verify that  
	\[
	\begin{pmatrix}
		0 & & & b_0 \\
		1 & & & b_1 \\
		& \ddots & & \vdots \\
		0 & & 1 & b_{n-1}
	\end{pmatrix}
	\begin{pmatrix}
		1 & & & & c_0 \\
		& 1 & & & c_1 \\
		& & \ddots & & \vdots \\
		& & & 1 & c_{n-2} \\
		& & & & c_{n-1}
	\end{pmatrix}
	=
	\begin{pmatrix}
		0 & & & b_0c_{n-1} \\
		1 & & & c_0 + b_1c_{n-1} \\
		& \ddots & & \vdots \\
		0 & & 1 & c_{n-2} + b_{n-1}c_{n-1}
	\end{pmatrix}.
	\]
	Comparing this expression with \( A \), it is necessary to verify whether \( b_0c_{n-1} = a_0 \) and \( c_{i-1} + b_ic_{n-1} = a_i \) for all \( i \in \{1, 2, \ldots, n-1\} \). Additionally, since we aim to express \( A \) as a product of two diagonalizable matrices, we require both factors in the product to be diagonalizable. By Lemma~\ref{cheo khoi}, this necessitates assuming \( c_{n-1} \neq 1 \). For the first factor, we will select \( b_0, b_1, \ldots, b_{n-1} \in F \) such that its characteristic polynomial has \( n \) distinct roots in \( F \).
	
	With the discussion outlined above, we proceed to consider the possibility. Since \( F \) contains a sufficiently large number of elements, there exist  \( b_0, b_1, \ldots, b_{n-1} \in F \) such that the characteristic polynomial of the matrix  
	\[
	\begin{pmatrix}
		0 & & & b_0 \\ 
		1 & & & b_1 \\ 
		& \ddots & & \vdots \\ 
		0 & & 1 & b_{n-1}
	\end{pmatrix}
	\]  
	has \( n \) distinct roots in \( F \). With these choices of \( b_i \)'s, the \( c_i \)'s are constructed as follows: let \( c_{n-1} = b_0^{-1}a_0 \), where we additionally assume \( b_0 \neq 0 \), which is a feasible condition given our assumptions. For each \( i \in \{1, 2, \ldots, n-1\} \), let \( c_{i-1} = a_i - b_ic_{n-1} \).  To ensure that \( c_{n-1} \neq 1 \), we assume \( b_0 \neq a_0 \), which is possible since \( F \) contains a sufficiently large number of elements.
	
	Drawing from the above discussion, we proceed to state the following result, which stands as one of the central contributions of this paper.
	
	\begin{theorem}\label{enough}
		Let $D$ be a division ring and let $n$ be a positive integer. If the center of $D$ contains sufficiently many elements, then every matrix in $\mathrm{M}_n(D)$ is expressible as a product of two diagonalizable matrices in $\mathrm{M}_n(D)$. Moreover, the number $2$ is sharp.
	\end{theorem}
	
	An alternative proof for Theorem~\ref{enough} can be found in \cite{Pa_DuSo_25}, though it is more complex than the one presented above. Since \cite{Pa_DuSo_25} is yet to be published, we provide a brief proof here for the sake of completeness.
	
	Let $F$ be the center of $D$. Since  $F$ contains sufficiently many elements, we can choose distinct elements $\lambda_1,\ldots,\lambda_n\in F$ and take $B=\mathrm{diag}(\lambda_1,\ldots,\lambda_n)\in\mathrm{M}_n(F)$. Now, let $A\in\mathrm{M}_n(D)$. The claim is trivial by the result of Botha if $A$ has the form: $A=\lambda\mathrm{I}_n$ where $\lambda\in F$. Hence, we focus on the case where $$A\notin\{\lambda\mathrm{I}_n\mid\lambda\in F\}.$$ First, we consider that $A\in\mathrm{GL}_n(D)$. Taking \cite[Theorem 2.1]{Pa_EgGo_19} into account,  there exists $P\in \mathrm{GL}_n(D)$ such that $$P^{-1}AP=UHV$$ where $U\in \mathrm{LT}_n(D), V\in \mathrm{UT}_n(D)$ and $H=\mathrm{diag}(1,\dots,1,h)$ for some $h$ in $D\setminus\{0\}$. Combining the observation $P^{-1}AP=(UB)(B^{-1}HV)$ with a similar argument in \cite[Lemma 2.1]{Pa_DuSo_25_1}, we conclude that both $UB$ and $B^{-1}HV$ are diagonalizable matrices in $\mathrm{M}_n(D)$; hence  $A$ is a product of two diagonalizable matrices, as promised. The remaining case is that $A\notin\mathrm{GL}_n(D)$. In this case, if $A$ is nilpotent, then this completes the proof by Theorem~\ref{nil2}. Our final focus is on the case where $A$ is not nilpotent. According to  \cite[Theorem 15, Page 28]{Bo_Ja_43}, there exists $P\in\mathrm{GL}_n(D)$ such that $$P^{-1}AP=\begin{pmatrix}
		A_1&0\\0&A_2
	\end{pmatrix}$$ where $A_1\in\mathrm{GL}_{n-t}(D)$ and $A_2\in\mathrm{M}_{t}(D)$ is nilpotent for some positive integer $t$. Following the arguments above, we conclude that  $A_1$ and $A_2$ can each be represented as products of two diagonalizable matrices in $\mathrm{M}_{n-t}(D)$ and $\mathrm{M}_t(D)$, respectively, implying $A\in\mathrm{M}_n(D)$ can also be represented similarly. These assertions complete the proof.	
	
	For the readers' convenience and to facilitate referencing the upcoming section, we conclude this section by briefly summarizing the results obtained so far.
	
	\begin{theorem}\label{section diagonalizable}
		Let $D$ be a  division ring with center $F$, and let $n$ be a positive integer. Then, every matrix $A$ in $\mathrm{M}_n(D)$ can be expressed as a sum  of $t$ or a product of $\ell$ diagonalizable matrices in $\mathrm{M}_n(D)$, where
		\begin{enumerate}[\rm (i)]
			\item $t=2$ in either of the following cases:
			\begin{enumerate}[\rm (a)]
				\item 	$F$ contains at least three elements.
				\item $A$ is nilpotent.
				\item $n=2$ and $D$ is noncommutative.
			\end{enumerate}
			\item $t=3$ if $F$ is the field of two elements.
			\item $\ell=2$ in either of the following cases:
			\begin{enumerate}[\rm (a)]
				\item $A$ is nilpotent.
				\item The characteristic of $D$ is different from $2$, and $D$ is not the field of three elements.
				\item The characteristic of $D$ is $2$, and $D=F$ is not the field of two elements.
				\item $F$ contains sufficiently many
				elements.
			\end{enumerate}
			\item $\ell=3$ if $D$ is the field of three elements.
			\item $\ell=4$ if the characteristic of $D$ is $2$, and $D\neq F$.
		\end{enumerate}Moreover, in the case where $D=F$ is the field of two elements, the nonsingular matrix in $\mathrm{M}_n(D)$, which is expressible as a product of diagonalizable matrices in $\mathrm{M}_n(D)$, is only $\mathrm{I}_n$.
	\end{theorem}

	\section{Some Waring-type results for images of polynomials}\label{Waring section}
	
	\subsection{Division rings with radicable multiplicative groups}\label{radicable}
	
	For convenience to describe our results, we begin by recalling the concept of a division ring with radicable multiplicative group, along with its motivation. A question in the theory of non-commutative rings is identifying which groups can serve as the multiplicative group of a division ring. This question originates from the work of Wedderburn, who famously proved that every finite division ring is a finite field. Kaplansky later extended this by showing that every center-by-periodic division algebra is commutative. Additionally, Hua demonstrated that the multiplicative group of any non-commutative division ring is insoluble, highlighting its inherent non-commutative nature. These results have been further generalized to include subnormal subgroups, as well as maximal and irreducible subgroups of the multiplicative group of division rings. See \cite{Pa_Ha_14} in detail.
	
	In 2011, Mahdavi-Hezavehi and Motiee introduced in \cite{Pa_Ma_11} the study of division rings whose multiplicative groups are radicable, which are referred to as \textit{radicable division rings}. In group theory, a group \( G \) is \textit{radicable} if, for every \( a \in G \) and each positive integer \( k \), there exists \( b \in G \) such that \( a = b^k \).  With this in mind, a division ring \( D \) is called \textit{radicable} if its multiplicative group, \( D^\times=D \setminus \{0\} \), is radicable. Equivalently, \( D \) is radicable if every equation of the form \( x^k - a = 0 \), where \( a \in D \) and \( k \) is a positive integer, has a solution in \( D \). Notably, it is straightforward to verify that such a \( D \) must either be an infinite division ring or the field with two elements. Indeed, if \( D \) is a finite division ring with \( k \) denoting the number of nonzero elements of $D$, and if \( a \in D \setminus \{0\} \), then by the definition of radicable division rings and the properties of \( k \), it follows that \( a = b^k = 1 \) for some \( b \in D \setminus \{0\} \), implies that \( D = \{0, 1\} \), thereby confirming the claim. It is natural to wonder whether the center of a radicable division ring must either be infinite or consist of only two elements.
	
	However, there exists a radicable division ring with finite center; here the center is not necessarily the field of two elements. Such an example can be constructed as follows. According to \cite{Pa_Co_73} and page 308 of \cite{Bo_Co_95}, an \emph{existentially closed division ring} over a field $k$ (or \emph{EC-division ring}, for short) is a division ring $D$ that is a $k$-algebra, satisfying the property that any existential sentence with constants from $D$, which holds in some extension of $D$, already holds in $D$ itself. For instance, if two matrices are similar over an extension of an EC-division ring $D$, then they are also similar over $D$ itself.  
	
	It is shown in \cite[Theorem 6.5.3]{Bo_Co_95} that if $K$ is a division ring with $k$ as a central subfield, then there exists an EC-division ring $D$ (over $k$) containing $K$, in which every finite consistent set of equations over $K$ has a solution. Here, a consistent set of equations is one that has a solution in some \textbf{extension} division ring. Furthermore, by \cite[Corollary 6.5.6]{Bo_Co_95}, the center of an EC-division ring $D$ over $k$ is precisely $k$.  
	
	Additionally, by \cite[Theorem 8.5.1]{Bo_Co_95}, any left-sided polynomial equation in one variable with coefficients in a division ring has a solution in some extension of the division ring. In particular, the equation $x^m - a = 0$, where $m$ is a positive integer and $a \in D$, is always consistent in the above sense. Hence, $D$ is a radicable division ring. To obtain an example with finite center, we choose $k$ to be a finite field, such as a prime subfield of $K$ of positive characteristic.  
	
	Notably, the assumption that every finite consistent set of equations over $K$ has a solution is not required in this construction; rather, what is essential is the existence of an EC-division ring over a finite field. We are grateful to a user on Mathematics Stack Exchange, in Question 5030291, for bringing this result to our attention.
	
	In summary, let $K$ be a division ring of positive characteristic. We then can choose $k$ as a finite field in the center of $K$. From Cohn’s book, we know that there exists an EC-division ring, denoted by $D$, that contains $K$. Furthermore, the equation $x^m - a = 0$, where $m$ is a positive integer and $a \in D$, always has a solution in some extension of $D$. Additionally, since $D$ is an EC-division ring, the equation $x^m - a = 0$ always has a solution within $D$ itself. Moreover, the center of $D$ is exactly $k$, which gives the desired example.
	
	In the finite-dimensional case, radicable noncommutative division rings are well understood in \cite[Theorem 3.18]{Pa_Ma_11} and at the end of the paper \cite{Pa_Ma_11}. Anyway, we only need to focus on the fact that it always has characteristic $0$, and such rings with indivisible center $F$ are precisely ordinary quaternion division rings and $F$ is a radically real-closed field. However, the structure of infinite-dimensional radicable division rings remains largely unexplored and represents an open problem.

	By combining the definition of radicable  division rings with Theorem~\ref{section diagonalizable}, we now present the following statement concerning \( k \)-th powers, which can be viewed as a version of the classical Waring problem for matrix rings. Since the correctness of this statement can be easily verified, we omit the proof. Moreover, in the spirit of the classical Waring problem, we state the result for the best bound, which is the number $2$.
	
	\begin{theorem}\label{diagonal map}
		Let $D$ be a radicable division ring with infinite center, and let $n$ be a positive integer. If $k$ is a positive integer, then every matrix  in $\mathrm{M}_n(D)$ can be expressed as a sum or a product  of two $k$-powers in $\mathrm{M}_n(D)$.
	\end{theorem}
	
	An ordinary quaternion division ring \( D \) with  real-closed center is known to be a radicable division ring with  infinite center. Notably, it was demonstrated in \cite[Proposition 3.7]{Pa_Fa_18} that every nonsingular matrix in \( \mathrm{M}_n(D) \) can be expressed as a \( k \)-power for any positive integer \( k \). On the other hand, nilpotent matrices do not generally possess this property, even in the case where \( k = 2 \).

	Building on Theorem~\ref{diagonal map}, it is straightforward to verify that every matrix over \( D \) can be expressed as a linear combination of two \( k \)-powers over \( D \), with the two scalars chosen arbitrarily, as long as they are nonzero. Notably, Theorem~\ref{diagonal map}, drawing inspiration from the research projects of A. Singh, provides a framework for investigating the image of a polynomial map known as the diagonal map, introduced in a preprint on arXiv \cite{Pa_Pa_24_1}. It also offers an intriguing perspective in the broader context of noncommutative division rings. As the validity of this statement is readily verifiable, we omit the proof and present the result below. To set the stage, we first recall the concept of diagonal polynomials. Let \( D \) be a division ring with center \( F \). Given any positive integers \( m, k_1, k_2, \dots, k_m \) and any nonzero scalars \( \lambda_1, \lambda_2, \dots, \lambda_m \in F \), we consider the polynomial  
	\[
	f(x_1, x_2, \dots, x_m) = \lambda_1 x_1^{k_1} + \lambda_2 x_2^{k_2} + \dots + \lambda_m x_m^{k_m},
	\]
	which is commonly referred to as a \textit{diagonal} polynomial in the non-commuting variables \( x_1, x_2, \dots, x_m \). As is well known, this polynomial induces the mapping  
	\[
	\widetilde{f} : \mathrm{M}_n(D)^m \to \mathrm{M}_n(D), \quad (A_1, A_2, \dots, A_m) \mapsto f(A_1, A_2, \dots, A_m).
	\]  
	A natural question is to determine the size of its image, particularly whether \( \widetilde{f} \) is surjective, and to find the smallest value of \( m \) for which this holds.  
	
	The study of such polynomials has attracted considerable interest. In the special case where \( k_1 = k_2 = \dots = k_m = k \), the polynomial \( f \) is known as a \textit{\( k \)-form}. Such a \( k \)-form is said to be \textit{universal} on \( \mathrm{M}_n(D) \) if the map \( \widetilde{f} \) is surjective. A notable instance of this problem is the classical Waring-type problem for matrices, which considers the case \( \lambda_1 = \lambda_2 = \dots = \lambda_m = 1 \) for a \( k \)-form and seeks the smallest \( m \) for which the form is universal, i.e., \( \widetilde{f} \) is surjective.  
	
	The universality of quadratic forms (the case \( k = 2 \)) over fields is a well-studied problem, with significant connections to arithmetic aspects. Notable contributions in this direction include the works of Bhargava \cite{Bo_Bha_00, Pa_Bha_16} and Voloch \cite{Pa_Vo_85}. More broadly, problems concerning the images of polynomial maps on algebras have been investigated by Kanel-Belov, Yavich, Kunyavskii, Rowen, and others, with a survey of recent developments available in \cite{Pa_Ka_20}.  
	
	With this background in place, we now present the following theorem, as mentioned earlier.

	\begin{theorem}\label{dia1}
		Let $D$ be a radicable division ring with infinite center, and let $n$ be a positive integer. If $k$ is a positive integer, then every matrix in $\mathrm{M}_n(D)$ can be expressed as a linear combination  of two $k$-powers in $\mathrm{M}_n(D)$, with the two scalars chosen arbitrarily but must be nonzero.  Moreover, the map induced by a diagonal polynomial in $m$ non-commuting variables, as described above,  is surjective for all $m\geq2$.
	\end{theorem}

	Regarding the Waring problem in groups, the study of products of square elements has a long history, possibly beginning with the work of R. Lyndon, T. McDonough, and M. Newman in 1973 \cite{Pa_Ly_73, Pa_Ly_73_1}. Since then, further developments have been made in \cite{Pa_Ak_01, Pa_Ak_02, Pa_Ak_10, Pa_Sa_04, Pa_Li_12,Pa_Bi_25}, continuing to the present day.  Waring problem for the special linear group \( \mathrm{SL}_n(D) \), where \( D \) is a division ring with center \( F \) of characteristic different from $2$ such that \( \dim_F D \leq 4 \), is studied in \cite{Pa_Bi_25}. It is shown that every element of \( \mathrm{SL}_n(D) \) can be expressed as a product of at most three square elements. As an application, let \( FG \) be the group algebra of a locally finite group \( G \) over a field \( F \) of characteristic \( p \neq 2 \). It is demonstrated that if either \( p > 2 \), \( F \) is algebraically closed, \( F \) is real-closed, or \( G \) is locally nilpotent, then every element in the derived subgroup \( \mathcal{U}(FG)' \) of the unit group $\mathcal{U}(FG)$ of $FG$ can be written as a product of at most three squares in $\mathcal{U}(FG)$. Furthermore, some of the results are also valid for twisted group algebras of locally finite groups.

	With these results  and the findings so far, we first show that for a division ring \( D \) with center, containing a real-closed field $R$ and satisfying $\dim_RD<\infty$, every matrix in \( \mathrm{SL}_n(D) \) can be written as a product of  two square elements in \( \mathrm{GL}_n(D) \). Extending this to twisted group algebras, under certain conditions, we then consider the case where \( F^\tau G \) is the twisted group algebra of a locally finite group \( G \) over a real-closed field \( F \) and prove that every element in the derived subgroup \( \mathcal{U}(F^\tau G)' \) of the unit group \( \mathcal{U}(F^\tau G) \) can also be expressed as a product of  two squares in \( \mathcal{U}(F^\tau G) \). This result strengthens a special case by reducing the bound from $3$ to $2$. Furthermore, it follows the spirit of \cite{Pa_Li_12}, which established that every element of every finite simple group is a product of two squares, including certain simple linear groups when \( D \) is a finite field.
	
	To state this precisely, we first recall some essential definitions and results. The notations \( \mathrm{GL}_n(D) \) and \( \mathrm{SL}_n(D) \) were introduced earlier in Subsection~\ref{notation} of Section~\ref{intro}. A \textit{locally finite} group is a group where every finitely generated subgroup is finite. We now turn to twisted group algebras.  
	
	Let $F$ be a field and let  $G$ be a group. Then a \textit{twisted group algebra} $F^\tau G$ of $G$ over $F$ is an associative ring which contains $F$ and has as an $F$-basis the set ${\overline G}$, a copy of $G$. Thus each element of $F ^\tau G$ is uniquely a finite sum $\sum_{x\in G} r_x{\bar x}$ with $r_x\in F$. Addition is as expected and multiplication is determined by the rule ${\bar x}{\bar y}=\tau(x,y) \overline{xy}$, where   $x,y\in G$ and  $\tau :G\times G\to F\setminus\{0\}$ is {called} the {\it twisting} function. The unit group of $F^\tau G$, consisting of all invertible elements of \( F^\tau G \), is denoted by \( \mathcal{U}(F^\tau G) \). We also use \( \mathcal{U}(F^\tau G)' \) to represent the subgroup of \( \mathcal{U}(F^\tau G) \) generated by the multiplicative commutators of \( \mathcal{U}(F^\tau G) \). By the multiplicative commutator of elements \( a \) and \( b \) in a group, we mean the element \( aba^{-1}b^{-1} \).

	Recall that a field $F$ is \textit{real-closed} if $F$ is not algebraically closed but the field extension $F(\sqrt{-1})$ is algebraically closed.  A division ring $D$ with center $F$ is called an \textit{ordinary quaternion division ring} if there exist $i,j,k\in D$ such that $D$ has the form $$D=\{a+bi+cj+dk\mid a,b,c,d\in F,i^2=j^2=k^2=-1,ij=-ji=k\}.$$ 
	
	With the information introduced, it is now time to present the following result.
	
	\begin{theorem}\label{real square}
		Let $D$ be a division ring with center, and let $n$ be a positive integer. If $D$ contains a real closed field $R$ such that $\dim_RD<\infty$, then every element  in $\mathrm{SL}_n(D)$ is a product of two square elements in  ${\rm GL}_n(D)$. 
	\end{theorem}
	
	\begin{proof}
		It is shown in \cite[Theorem]{Pa_Ge_60} that if \( D \) is a division ring containing a real-closed subfield \( R \) such that \( D \) has finite left dimension over \( R \), then \( D \) must be one of the following:  
		\begin{enumerate}[\rm (i)]
			\item \( D = R \);
			\item \( D = R(\sqrt{-1}) \);
			\item \( D \) is an ordinary quaternion division ring with  real-closed center \( R' \) satisfying \( R'(\sqrt{-1}) \cong R(\sqrt{-1}) \).
		\end{enumerate} If \( n = 1 \) and \( D = R \) or \( D = R(\sqrt{-1}) \), then \( \mathrm{SL}_1(D) = \{1\} \), so every element in \( \mathrm{SL}_1(D) \) is trivially a product of two square elements in \( \mathrm{GL}_1(D) \). If \( n = 1 \) and \( D \) is a quaternion division ring with real-closed center, then by \cite[(16.14) Theorem (Niven, Jacobson)]{Bo_La_91}, every element in \( D \) can be represented as a square. In particular, each element of \( \mathrm{SL}_1(D) \) is a square, and therefore, every element of \( \mathrm{SL}_1(D) \) can be expressed as a product of two square elements in \( \mathrm{GL}_1(D) \). By the way, by a similar argument to \cite[Lemma 2.1]{Pa_DuHaSo_24}, we can describe \( \mathrm{SL}_1(D) \) as the set $$ \{ a + bi + cj + dk \mid a, b, c, d \in R, a^2 + b^2 + c^2 + d^2 = 1 \}.$$ Thus, the proof for the case where \( n = 1 \) is complete. We now proceed to the case where \( n > 1 \).

		Let $A\in\mathrm{SL}_n(D)$ and we will now proceed by dividing the proof into two cases, as outlined above.
		
		\medskip
		
		\noindent\textbf{Case 1.} $D$ is a field. In this case,
		the characteristic of $D$ is $0$. Furthermore, it is shown in \cite[Theorem 1]{Pa_SoDuHaBi_2021} that $A$ is expressible as a product of two commutators of involutions in $\mathrm{GL}_n(D)$. One can easily verify that any commutator of involutions in \( \mathrm{GL}_n(D) \) is a square element in \( \mathrm{GL}_n(D) \), as it follows from the relation $$ aba^{-1}b^{-1} = abab = (ab)^2, $$ given that \( a^2 = b^2 = 1 \). Therefore, $A$ can be expressed as a product of two square elements in $\mathrm{GL}_n(D)$.
		
		\medskip

		\noindent\textbf{Case 2.} $D$ is an ordinary quaternion division ring with real-closed center $R'$. Taking Theorem~\ref{section diagonalizable} into account, it follows that $A$ can be expressed as a product of two diagonalizable matrices in $\mathrm{GL}_n(D)$. Furthermore, by \cite[(16.14) Theorem (Niven, Jacobson)]{Bo_La_91}, every element of $D$ is representable as a square element in $D$. Hence, $A$ can be expressed as a product of two square elements in $\mathrm{GL}_n(D)$.
		
		\medskip
		
		In conclusion, the proof follows from the two cases discussed above.
	\end{proof}
	
	Taking Theorem~\ref{real square} into account, we now turn our attention to verifying the following theorem concerning group algebras, thus bringing this section to a close.
	
	\begin{theorem}\label{twisted}
		Let $F^\tau G$ be the twisted group algebra of a locally finite group $G$ over a field $F$ of characteristic $0$. If $F$ is either algebraically closed or is real-closed, then every element in $\mathcal{U}(F^\tau G)'$ is a product of two square elements in $\mathcal{U}(F^\tau G)$.
	\end{theorem}
	
	\begin{proof}
		Assume that $$\alpha=(\alpha_1\beta_1\alpha_1^{-1}\beta_1^{-1})\cdots(\alpha_k\beta_k\alpha_k^{-1}\beta_k^{-1})\in\mathcal{U}(F^\tau G)',$$ where $\alpha_i,\beta_i\in \mathcal{U}(F^\tau G)$. Let $H$ be a subgroup of $G$ generated by the supports of all the $\alpha_i,\beta_i,\alpha_i^{-1}$ and  $\beta_i^{-1}$. Since $G$ is a locally finite group, it follows that $H$ is a finite group and  $\alpha\in \mathcal{U}(F^\tau H)'$. 
		
		If \( F \) has characteristic 0, then by \cite[Theorem~4.2]{Bo_Pa_89}, it follows that \( F^\tau H \) is a semisimple \( F \)-algebra. It is well known from  the Wedderburn-Artin theorem that
		\[
		F^\tau H \cong \mathrm{M}_{n_1}(D_1) \times \dots \times \mathrm{M}_{n_t}(D_t),
		\]
		where \( t, n_1, \dots, n_t \) are positive integers, and \( D_1, \dots, D_t \) are division rings that are finite-dimensional over \( F \).
		
		If \( F \) is algebraically closed, then by \cite[Proposition 1.6]{Bo_Bre_14}, we conclude that $$ D_1 = \cdots = D_t = F. $$ Following Case 1 in the proof of Theorem~\ref{real square}, it is shown that every matrix over \( F \) can be expressed as a product of two square elements. Therefore, the same holds for \( \alpha \in \mathcal{U}(F^\tau H)' \).
		
		In the case where \( F \) is a real-closed field, the proof follows in a similar manner to the algebraically closed case. However, here we directly apply Theorem~\ref{real square} without referring to Case 1 of its proof, which allows us to complete the argument, as promised.
	\end{proof}
	
	\subsection{Algebraically closed division rings}\label{algebraically}
	
	In the preceding subsubsection, our focus was on radicable division rings. Here, we examine a particular class of such rings. Recall that a division ring \( D \) is said to be radicable if every equation of the form \( x^k - a = 0 \), where \( a \in D \) and \( k \) is a positive integer, has a solution in \( D \). A stronger condition is that every polynomial equation in a \textbf{central} variable with coefficients in \( D \) admits a right (or left) root in \( D \), meaning that \( D \) is right (or left) algebraically closed. Specifically, a polynomial equation in the \textbf{central} variable \( x \) takes the form  
	\[
	a_t x^t + a_{t-1} x^{t-1} + \cdots + a_1 x + a_0 = 0,
	\]  
	where \( t \) is a positive integer, \( a_t, a_{t-1}, \dots, a_1, a_0 \in D \) with \( a_t \neq 0 \). It is important to emphasize that the nonconstant polynomial  
	\[
	a_t x^t + a_{t-1} x^{t-1} + \cdots + a_1 x + a_0
	\]  
	can also be expressed as  
	\[
	x^t a_t + x^{t-1} a_{t-1} + \cdots + x a_1 + a_0.
	\]  
	However, evaluating these expressions at a given element may yield different results. Accordingly, an element \( a \in D \) is called a \textit{right} (respectively, \textit{left}) \textit{root} of the polynomial if  
	\[
	a_t a^t + a_{t-1} a^{t-1} + \cdots + a_1 a + a_0 = 0 \quad \text{(resp. } a^t a_t + a^{t-1} a_{t-1} + \cdots + a a_1 + a_0 = 0\text{).}
	\]  
	Since we primarily consider right roots in this text, we shall often omit the qualifier ``right" and simply refer to them as roots.  The case of  non-central variables will be considered later.
	
	In the finite-dimensional setting, the notions of ``left algebraically closed" and ``right algebraically closed" coincide, and algebraically closed division rings admit a well-known classification. Specifically, if \( D \) is a  division ring of finite-dimensional over its center, it follows from \cite[(16.16) Theorem]{Bo_La_91} that if \( D \) is algebraically closed, then precisely one of the following holds:  
	\begin{enumerate}[\rm (i)]
		\item \( D \) is an algebraically closed field.
		\item The center  of \( D \) is a real-closed field, and \( D \) is an ordinary quaternion division ring.
	\end{enumerate} Hence, the classification of these division rings is now fully determined, having been reduced to the classification of real-closed fields.  However, the structure of algebraically closed division rings of infinite dimensional over their centers remains less well understood. 
	
	As shown in Subsubsection~\ref{radicable}, we have provided an example of a radicable division ring whose center is a finite field. Using similar reasoning, we can construct an example demonstrating the existence of an algebraically closed division ring with finite center. Notably, such a division ring must be of infinite dimension over its center. Hence, in some subsequent results, it will be necessary to assume that the center of an algebraically closed division ring is infinite.

	As mentioned earlier, we primarily consider polynomials over a field. By applying the definition of algebraically closed division rings, one can readily verify the following theorem. Although it has been established in \cite{Pa_Pa_25}, we present it here for the reader’s convenience, as the result has not yet been published.  
	
	\begin{theorem}  \label{central x}
		Let \( D \) be an algebraically closed division ring with center \( F \). If \( p \) is a nonconstant polynomial in the central variable \( x \) with coefficients in \( F \), then the image of \( p \) on \( D \) coincides with \( D \), that is, \( p(D) = D \).  
	\end{theorem}  
	
	\begin{proof}  
		Let \( p \) be a nonconstant polynomial in the central variable \( x \) with coefficients in \( F \). We need to show that every element \( d \in D \) belongs to \( p(D) \). If \( D \) is a field, then it  is obviously an algebraically closed field. It is well known that any algebraically closed field is infinite. Moreover, by Theorem~\ref{Wedderburn}, every finite division ring is a field. Consequently, \( D \) must be infinite.  
		
		Furthermore, it follows from \cite[(16.7) Theorem, p. 251]{Bo_La_91} that no nonzero polynomial in the central variable \( x \) with coefficients in \( D \) can vanish identically on \( D \). Now, consider the polynomial \( p - d \), which remains a nonconstant polynomial in the central variable \( x \) with coefficients in \( D \), though not necessarily in \( F \). Since \( D \) is algebraically closed, this polynomial must have a root in \( D \), say \( a \), satisfying \( p(a) = d \). This completes the proof.  
	\end{proof}  
	
	Using Theorem~\ref{central x}, we will prove the result concerning polynomials in noncommuting variables, as stated below. The subsequent theorem has been established in \cite{Pa_Pa_25}, but, as it has not yet been published, we present it here for the reader’s convenience, for the same reason as previously mentioned in Theorem~\ref{central x}. The proof is derived from \cite{Pa_Wa_21}, with some modifications. Moreover, the following theorem serves as a generalization of Theorem~\ref{central x}.
	
	\begin{theorem}\label{al}
		Let $D$ be an algebraically closed division ring with center $F$ and let $p$ be a polynomial in non-commuting variables $x_1,\ldots,x_m$ with coefficients in $F$, where $m$ is a positive integer, such that $p$ has zero constant term. If $p(F)\neq\{0\}$, then $p(D)=D$.
	\end{theorem}
	
	\begin{proof}
		The case where \( m=1 \) is exactly addressed by Theorem~\ref{central x}. Therefore, we now focus on the case where \( m \geq 2 \).
		
		Let \( p = p(x_1, \ldots, x_m) \) be a polynomial in \( m \) variables, where \( m \geq 2 \). Since the polynomial \( p \) has zero constant term, we immediately have
		\[
		p(0, \ldots, 0) = 0.
		\]
		Now, assume that \( p(F) \neq \{0\} \), meaning there exist elements \( a_1, \ldots, a_m \in F \) such that \( p(a_1, \ldots, a_m) \neq 0 \). If all the degrees of \( x_1, \ldots, x_m \) were zero, we would face the contradiction
		\[
		0 \neq p(a_1, \ldots, a_m) = p(0, \ldots, 0) = 0.
		\] This implies that there must be at least one index \( i \in \{1, \ldots, m\} \) for which the degree of \( x_i \) in \( p \) is nonzero. 
		
		Without loss of generality, assume the degree of \( x_1 \) in \( p \) is nonzero. Define the polynomial \( f(x) = p(x, a_2, \ldots, a_m) \), which is in the central variable \( x \). Since $$ f(a_1) = p(a_1, a_2, \ldots, a_m) \neq 0, $$ it follows that \( f(x) \) is a nonzero polynomial. For any \( c \in D \), following the proof of Theorem~\ref{central x}, we know that \( f(x) - c \) remains a nonconstant polynomial in the variable \( x \) with coefficients in \( D \). Given that \( D \) is algebraically closed, Theorem~\ref{central x} tells us that \( f(D) = D \), and consequently, \( p(D) = D \). This concludes the proof.
	\end{proof}

	By Theorem~\ref{al}, it follows that \( p(D) = D \), where \( p \) is a polynomial in non-commuting variables \( x_1, \dots, x_m \) with coefficients in \( F \), and \( m \) is a positive integer. Since \( p \) has zero constant term, it follows that the image of \( p \), evaluated on matrices over \( D \), must contain all diagonalizable matrices over \( D \). This can be seen by noting that the image of \( p \) is invariant under similarity transformations, as the coefficients lie in the center \( F \) of \( D \). Therefore, it suffices to show that the image of \( p \) includes all diagonal matrices over \( D \). We now proceed this as follows. Let \( A = \mathrm{diag}(a_1, a_2, \dots, a_n) \) be a diagonal matrix, where \( n \) is a positive integer and \( a_1, a_2, \dots, a_n \in D \). Since \( p(D) = D \), it follows that each \( a_i \) can be written as $$a_i = p(b_{1,i}, \dots, b_{m,i}),$$ for some \( b_{1,i}, \dots, b_{m,i} \in D \). Let \( B_i = \mathrm{diag}(b_{i,1}, b_{i,2}, \dots, b_{i,m}) \) for \( i \in \{1, 2, \dots, n\} \). It follows that \( A = p(B_1, \dots, B_m) \), as desired. Combining the claim with Theorem~\ref{section diagonalizable}, we derive the following theorem.
	
	\begin{theorem}\label{al1}
		Let $D$ be an algebraically closed division ring with infinite center $F$ and let $n$ be a positive integer. If $p$ is a polynomial in non-commuting variables $x_1,\ldots,x_m$ with coefficients in $F$, where $m$ is a positive integer, such that $p$ has zero constant term such that $p(F)\neq\{0\}$, then  every matrix  in $\mathrm{M}_n(D)$ can be expressed as a sum, or a product, or a linear combination of two elements from $p(\mathrm{M}_n(D))$, where the two scalars chosen arbitrarily in $F\setminus\{0\}$.
	\end{theorem}  
	
	Theorem~\ref{al1} also relates to certain results from the work of M. Brešar and P.~Šemrl \cite{Pa_Bre_23,Pa_Bre_23_1}, but in the context of non-commutative settings, as viewed through the lens of the modified assumptions in Theorem~\ref{al1}. We now conclude this subsubsection by presenting the following theorem, derived from \cite[Corollary~4.5]{Pa_Bre_23}. We say two polynomials $f$ and $g$ in non-commuting variables  are  \textit{cyclically equivalent} if the polynomial $f-g$ can be written as a sum of commutators.
	
	\begin{theorem}\label{allll}
		Let $D$ be an algebraically closed division ring with center $F$ of characteristic $0$ such that $\dim_FD<\infty$ and let $n\geq2$ be a positive integer. If $p$ is a polynomial in non-commuting variables $x_1,\ldots,x_m$, where $m$ is a positive integer, which is not cyclically equivalent to an identity of $\mathrm{M}_n(F)$, then every matrix in $\mathrm{M}_n(D)$ can be expressed as a linear combination of five elements from $p(\mathrm{M}_n(D))$.
	\end{theorem}
	
	\begin{proof}
		Since \( D \) is a  division ring of finite-dimensional over its center, it follows from \cite[(16.16) Theorem]{Bo_La_91} that if \( D \) is algebraically closed, then precisely one of the following holds:  
		\begin{enumerate}[\rm (i)]
			\item \( D \) is an algebraically closed field.
			\item The center $F$ of \( D \) is a real-closed field, and \( D \) is an ordinary quaternion division ring.
		\end{enumerate} The first case where \( D \) is an algebraically closed field is done by \cite[Corollary~4.5]{Pa_Bre_23}. Now, we turn our attention to the latter case. Following the proof of \cite[Proposition~3.7]{Pa_Fa_18}, we know that \( A \) has a Jordan normal form, meaning that there exists \( P \in \mathrm{GL}_n(D) \) such that  
		\[
		P^{-1} A P = \bigoplus_{i=1}^s J_{m_i}(\alpha_i),  
		\]  
		for some positive integers \( s, m_1, m_2, \ldots, m_s \), and scalars \( \alpha_1, \alpha_2, \ldots, \alpha_s \in F(\sqrt{-1}) \), where \( m_1 + m_2 + \cdots + m_s = n \). Here,  
		\[
		J_{m_i}(\alpha_i) = \begin{pmatrix} 
			\alpha_i & 1 & 0 & \cdots & 0 & 0 & 0 \\ 
			0 & \alpha_i & 1 & \cdots & 0 & 0 & 0 \\ 
			0 & 0 & \alpha_i & \cdots & 0 & 0 & 0 \\ 
			\vdots & \vdots & \vdots & \ddots & \vdots & \vdots & \vdots \\ 
			0 & 0 & 0 & \cdots & \alpha_i & 1 & 0 \\ 
			0 & 0 & 0 & \cdots & 0 & \alpha_i & 1 \\ 
			0 & 0 & 0 & \cdots & 0 & 0 & \alpha_i
		\end{pmatrix} \in \mathrm{M}_{m_i}(D).
		\]  
		Since the image of a polynomial is invariant under similarity transformations, it suffices to verify the conclusion for each \(\displaystyle \bigoplus_{i=1}^s J_{m_i}(\alpha_i) \). Note that $F(\sqrt{-1})$ is an algebrically closed field. Hence, taking \cite[Corollary~4.5]{Pa_Bre_23} into account, the proof is completed.
	\end{proof}

	\bibliographystyle{amsplain}

\end{document}